\documentclass[12pt,a4paper]{amsart}
\usepackage{amsfonts}
\usepackage{amssymb}
\usepackage{amsmath}
\RequirePackage{amsthm}

\newtheorem{theorem}{Theorem}[section]
\newtheorem{corollary}{Corollary}[section]

\newtheorem{remark}{Remark}[section]
\newtheorem{lemma}{Lemma}[section]

\title[A unified existence theory]{\bf A unified existence theory for
evolution equations and systems under nonlocal conditions}

\begin{document}

\maketitle

\noindent \centerline{\bf Tiziana Cardinali$^{\small a}$ - Radu Precup$^{\small b}$ - Paola Rubbioni$^{\small a,c}$
}

\vskip 1truecm
\noindent \centerline{\small $^{\small a}$ Department of Mathematics and Informatics, University of Perugia, Perugia, Italy}

\noindent \centerline{\small $^{\small b}$ Department of Mathematics, Babe\c{s}-Bolyai University, Cluj, Romania}

\noindent \begin{center}{\small $^{\small c}$ Corresponding author -
via L.Vanvitelli 1, 06125 Perugia (Italy)
- Phone: +390755855042}
\end{center}

\noindent \centerline{\small {\em E-mail addresses: tiziana@dmi.unipg.it - r.precup@math.ubbcluj.ro - rubbioni@dmi.unipg.it}}

\vskip 1truecm

\begin{abstract}
We investigate the effect of nonlocal conditions expressed by linear
continuous mappings over the hypotheses which guarantee the existence of
global mild solutions for functional-differential equations in a Banach
space. A progressive transition from the Volterra integral operator
associated to the Cauchy problem, to Fredholm type operators appears when
the support of the nonlocal condition increases from zero to the entire
interval of the problem. The results are extended to systems of equations in
a such way that the system nonlinearities behave independently as much as
possible and the support of the nonlocal condition may differ from one
variable to another.
\end{abstract}



\section{Introduction}

\renewcommand{\theequation}{\thesection
.\arabic{equation}}This paper deals with the Cauchy problem for
functional-differential evolution equations in a Banach space$\ X,\ $with a
nonlocal condition expressed by a linear mapping
\begin{equation}
\left\{
\begin{array}{l}
u^{\prime }\left( t\right) =A\left( t\right) u\left( t\right) +\Phi \left(
u\right) \left( t\right) ,\ \ \ \text{for\ a.a.\ }t\in \left[ 0,a\right] \\
u\left( 0\right) =F\left( u\right) .%
\end{array}%
\right.  \label{1}
\end{equation}%
Here $\left\{ A\left( t\right) \right\} _{t\in \left[ 0,a\right] }$ is a
family of densely defined linear operators (not necessarily bounded or
closed) in the Banach space $X$ generating an evolution operator, $\Phi $ is
a nonlinear mapping, and $F$ is linear.

Containing the general functional term $\Phi ,$ our equation is more general
than the most studied one given by%
\begin{equation*}
\Phi \left( u\right) \left( t\right) =g\left( t,u_{t}\right) ,
\end{equation*}%
where the function $u_{t}\left( s\right) =u\left( t+s\right) ,$ for $s\in
\lbrack -r,0],$ $r>0,$ $t\in \left[ 0,a\right] $ stands for the memory in
lots of models for processes with aftereffect (see, e.g. \cite{wu}). In
particular, it covers evolution equations which are perturbed by a
superposition operator $\Phi ,$
\begin{equation}
\Phi \left( u\right) \left( t\right) =f\left( t,u\left( t\right) \right) ,\
\ t\in \left[ 0,a\right]  \label{so}
\end{equation}%
associated to some function $f:\left[ 0,a\right] \times X\rightarrow X,$
integro-differential equations and equations with modified argument.

In the mathematical modeling of real processes from physics, chemistry or
biology, the nonlocal conditions can be seen as feedback controls by which
the "sum" of the states of the process along its evolution equals the
initial state. The mapping $F$ expressing the nonlocal condition can be
linear or nonlinear, of discrete or continuous type. For instance, as a
linear mapping, it can be given by a finite sum of multi-point form%
\begin{equation}
F\left( u\right) =\sum_{k=1}^{m}c_{k}u\left( t_{k}\right) ,  \label{mp}
\end{equation}%
where $0<t_{1}<t_{2}<...<t_{m}\leq a$ and $c_{k}$ are real numbers. More
general, it can be expressed in terms of a Stieltjes integral%
\begin{equation*}
F\left( u\right) =\int_{0}^{a}u\left( t\right) d\phi \left( t\right) .
\end{equation*}%
Nonlocal problems with multi-point conditions and more general with linear
and nonlinear nonlocal conditions were discussed in the literature by
various approaches. We refer the reader to the papers \cite{al}, \cite{bn}-%
\cite{cr}, \cite{f}, \cite{llx}, \cite{ll}, \cite{nt}, \cite{wi}, \cite{x}
and the references therein.

As it was first remarked in \cite{bp}, it is important to take into
consideration the \textit{support} of the nonlocal condition, that is the
minimal closed subinterval $\left[ 0,a_{F}\right] $ of $\left[ 0,a\right] $
with the property%
\begin{equation}
F\left( u\right) =F\left( v\right) \text{ whenever }u=v\text{ on }\left[
0,a_{F}\right] .  \label{supp}
\end{equation}%
This means that the mapping $F$ only depends on the restrictions of the
functions from $C\left( \left[ 0,a\right] ;X\right) ,$ to the subinterval $%
\left[ 0,a_{F}\right] .$ The case $a_{F}=0$ recovers the classical Cauchy
problem, while the case $a_{F}=a$ corresponds to a \textit{global} nonlocal
condition dissipated over the entire interval $\left[ 0,a\right] $ of the
problem. When $0<a_{F}<a,$ we say that the nonlocal condition is \textit{%
partial}. As we shall see, moving $a_{F}$ from $0$ to $a,$ we realize a
progressive transition from Volterra to Fredholm nature of the equivalent
integral equation.

The support problem is even more interesting in case of a system of
equations in $n$ unknown functions $u_{1},u_{2},...,u_{n},$ when a nonlocal
condition is expressed by a linear mapping $F=F\left(
u_{1},u_{2},...,u_{n}\right) .$ In this case, we may speak about the \textit{%
support of }$F$\textit{\ with respect to each of the variables}. The notion
is introduced in this paper for the first time, and together with the
vectorial method that is used, allows us to localize independently each
component $u_{i}$ of a solution $\left( u_{1},u_{2},...,u_{n}\right) .$

In addition, as an other original feature of our study, the localization of
a solution, and in case of systems, of each of the solution components, is
realized in a \textit{tube}, i.e. a set of the form%
\begin{equation*}
\left\{ \left( t,u\right) :\ t\in \left[ 0,a\right] ,\ u\in X,\ \left\vert
u\right\vert \leq R\left( t\right) \right\} ,
\end{equation*}%
of a time-depending radius $R\left( t\right) .$ In a physical
interpretation, this means that the variation of a quantity $u\left(
t\right) $ is allowed to be nonuniformly larger or smaller during the
evolution, as prescribed by function $R\left( t\right) .$

We finish this introductory part by some notations and basic results.
Throughout this paper, the norm of a Banach space $X$ is denoted by $%
\left\vert .\right\vert ,$ the open and closed balls of $X,$ of radius $R$
centered at the origin, are denoted by $B\left( 0,R\right) ,\ \overline{B}%
\left( 0,R\right) ,$ respectively; the symbol $\left\vert .\right\vert _{%
\mathcal{L}\left( X,Y\right) }$ is used for the norm of a linear continuous
mapping from $X$ to $Y,$ with the understanding that $\mathcal{L}(X,Y)$ is
the space of all bounded linear operators from $X$ to $Y.$ Also, the norm on
$L^{p}\left( b_{1},b_{2}\right) $ $\left( 1\leq p\leq \infty \right) $ is
denoted by $\left\vert .\right\vert _{L^{p}\left( b_{1},b_{2}\right) },$ and
the symbol $\left\vert .\right\vert _{L^{\infty }\left( b_{1},b_{2}\right) }$
is also used for the sup norm on $C\left[ b_{1},b_{2}\right]
:=C([b_{1},b_{2}];\mathbf{R}).$ The notation $L_{+}^{1}\left(
b_{1},b_{2}\right) $ stands for the set of all nonnegative functions in $%
L^{1}\left( b_{1},b_{2}\right) .$ The open and closed balls of $C\left( %
\left[ 0,a\right] ;X\right) $ of radius $R$ centered at the origin are
denoted by $B_{C}\left( 0,R\right) ,\ \overline{B}_{C}\left( 0,R\right) ,$
respectively.

We recall that an operator ~$T:\Delta \rightarrow \mathcal{L}(X,X)$, where $%
\Delta =\{(t,s):0\leq s\leq t\leq a\}$, is called an \emph{evolution operator%
} if $T(t,s):X\rightarrow X$ is a bounded linear operator for every $%
(t,s)\in \Delta ,$ and the following conditions are satisfied:

\begin{enumerate}
\item[(i)] $T(s,s)=I\ $(identity of $X),$ \quad $T(t,r)T(r,s)=T(t,s)$%
\thinspace\ for $0\leq s\leq r\leq t\leq a;$

\item[(ii)] $(t,s)\mapsto T(t,s)$\thinspace\ is strongly continuous on $%
\Delta .$
\end{enumerate}

\noindent Note that, since $T$ is strongly continuous on the compact set $%
\Delta ,$ there exists a constant $M>0$ such that
\begin{equation}
|T(t,s)|_{\mathcal{L}(X,X)}\leq M\,,\quad \text{for all\ \ }(t,s)\in \Delta .
\label{M}
\end{equation}

By $\alpha $ we shall denote the \emph{Kuratowski measure of noncompactness}
on a Banach space $X,$ i.e.%
\begin{equation*}
\alpha \left( D\right) =\inf \left\{ \varepsilon >0:D\text{ admits a finite
cover by sets of diameter }\leq \varepsilon \right\}
\end{equation*}%
for any bounded $D\subset X.$ The symbol $\alpha _{C}$ will stand for the
corresponding Kuratowski measure of noncompactness on $C\left( \left[
b_{1},b_{2}\right] ;X\right) .$ Recall (see \cite{a}, \cite{bg}, \cite{d},
\cite{koz}) that for an equicontinuous set $D\subset C\left( \left[
b_{1},b_{2}\right] ;X\right) $ with $D\left( t\right) $ bounded for each $%
t\in \left[ b_{1},b_{2}\right] ,$ one has%
\begin{equation}
\alpha _{C}\left( D\right) =\max_{t\in \left[ b_{1},b_{2}\right] }\alpha
\left( D\left( t\right) \right) .  \label{a1}
\end{equation}%
Also recall (see \cite{h}, \cite{op1}) that for a countable set $D\subset
L^{1}\left( b_{1},b_{2};X\right) $ with $\left\vert u\left( t\right)
\right\vert \leq \eta \left( t\right) $ for a.a. $t\in \left[ b_{1},b_{2}%
\right] $ and every $u\in D,$ where $\eta \in L_{+}^{1}\left(
b_{1},b_{2}\right) ,$ the function $t\mapsto \alpha \left( D\left( t\right)
\right) $ belongs to $L^{1}\left( b_{1},b_{2}\right) $ and
\begin{equation}
\alpha \left( \left\{ \int_{b_{1}}^{b_{2}}u\left( s\right) ds:u\in D\right\}
\right) \leq 2\int_{b_{1}}^{b_{2}}\alpha \left( D\left( s\right) \right) ds.
\label{h1}
\end{equation}%
The main tool of nonlinear functional analysis that we shall use is the
Leray-Schauder type continuation theorem of M\"{o}nch \cite{m} (see also
\cite{d}, \cite{op2}) involving a compactness condition which in particular
holds for condensing operators.

\begin{theorem}
\label{th 1.1}Let $U$ be an open subset of a Banach space $X,$ and let $N:%
\overline{U}\rightarrow X$ be continuous. Assume that for some $u_{0}\in U$
the following conditions are satisfied:

\emph{(a)} $N\left( u\right) -u_{0}\neq \lambda \left( u-u_{0}\right) $ on $%
\partial U$ for all $\lambda >1;$

\emph{(b)} if $C\subset \overline{U}$ is countable and $C\subset \overline{%
\text{\emph{conv}}}\left( \left\{ u_{0}\right\} \cup N\left( C\right)
\right) ,$ then $\overline{C}$ is compact.

Then $N$ has a fixed point in $\overline{U}.$
\end{theorem}

Finally, for the last part of the paper devoted to systems, we recall that
for a square matrix of nonnegative entries $H\in \mathcal{M}_{n\times
n}\left( \mathbf{R}_{+}\right) ,$ the \textit{spectral radius} $\rho \left(
H\right) $ is the maximum modulus of the eigenvalues, and that the following
statements are equivalent:

(i) $\ \rho \left( H\right) <1;$

(ii) $\ H^{k}\rightarrow 0$ (zero matrix) as $k\rightarrow \infty ;$

(iii) $\ I-H$ is nonsingular and the entries of $\left( I-H\right) ^{-1}$
are nonnegative ($I$ being the unit matrix of the same order).

Details can be found in \cite{v}.

\section{Existence and localization of solutions for evolution equations}

\setcounter{equation}{0}Compared to other papers on the existence of
solutions for local or nonlocal problems, our approach is to find solutions
in a `ball' of a time-depending radius. Hence we are looking for solutions
in the bounded closed subset of $C\left( \left[ 0,a\right] ;X\right) ,$%
\begin{equation*}
\overline{U}:=\left\{ u\in C\left( \left[ 0,a\right] ;X\right) :\left\vert
u\left( t\right) \right\vert \leq R\left( t\right) \text{ for all }t\in %
\left[ 0,a\right] \right\} ,
\end{equation*}%
where $R\in C\left[ 0,a\right] $ is a given function with $R\left( t\right)
>0$ for all $t\in \left[ 0,a\right] ,$ and
\begin{equation*}
U:=\left\{ u\in C\left( \left[ 0,a\right] ;X\right) :\left\vert u\left(
t\right) \right\vert <R\left( t\right) \text{ for all }t\in \left[ 0,a\right]
\right\} .
\end{equation*}%
In this section, the linear part of the equation of problem (\ref{1}) will
satisfy the following property (see, e.g. \cite{cr0}):

\begin{description}
\item[(A)] $\{A(t)\}_{t\in \lbrack 0,a]}$ is a family of linear not
necessarily bounded operators $(A(t):D(A)\subset X\rightarrow X,$ $t\in
\lbrack 0,a],$ $D(A)$ is a dense subset of $X$ not depending on $t)$
generating a continuous evolution operator $T:\Delta \rightarrow \mathcal{L}(X,X)$.
\end{description}

We shall assume that

\begin{description}
\item[(h1)] $\Phi :\overline{U}\rightarrow L^{1}(0,a;X)$ is continuous;

\item[(h2)] $F:$ $C\left( \left[ 0,a\right] ;X\right) \rightarrow X$ is a
linear continuous mapping such that the operator from $X$ to $X,$ $x\mapsto
x-F\left( T\left( .,0\right) x\right) $ has an inverse $B.$
\end{description}

Note that, by (h2) and the definition of the evolution operator, the
operator $B$ is linear and bounded, i.e. $B\in \mathcal{L}\left( X,X\right) $
(see \cite[Corollary 3.2.8]{dmp}).

\begin{remark}
\emph{A sufficient condition for (h2) to hold is that the norm of the
operator }$FT\left( .,0\right) $\emph{\ from }$X$\emph{\ to }$X$\emph{\ is
less than one. Indeed, in this case, }$FT\left( .,0\right) $\emph{\ is a
contractive mapping and consequently, the operator from }$X$\emph{\ to }$X,$%
\emph{\ }$x\mapsto x-F\left( T\left( .,0\right) x\right) $\emph{\ is
invertible. In the particular case, where }$F$\emph{\ is of discrete type,
given by (\ref{mp}), one has }$a_F=t_{m},$\emph{\ and the norm of the }$%
FT\left( .,0\right) $\emph{\ is less than one if}%
\begin{equation*}
M\sum_{k=1}^{m}\left\vert c_{k}\right\vert <1.
\end{equation*}
\end{remark}

Under conditions (h1) and (h2), a mild solution of the problem (\ref{1}) in $%
\overline{U}$ is a function $u\in \overline{U}$ such that%
\begin{eqnarray}
u\left( t\right) &=&T\left( t,0\right) BF\left( \int_{0}^{.}T\left(
.,s\right) \Phi \left( u\right) \left( s\right) ds\right)  \label{f1} \\
&&+\int_{0}^{t}T\left( t,s\right) \Phi \left( u\right) \left( s\right) ds,\
\ \ \text{for all \ }t\in \left[ 0,a\right] .  \notag
\end{eqnarray}

From now on, we shall denote by $\left[ 0,a_{F}\right] $ the support of $F.$
It is important to note that one has
\begin{equation*}
F\left( v\right) =F\left( \chi _{a_{F}}\left( v\right) \right) ,
\end{equation*}%
for all $v\in C\left( \left[ 0,a\right] ;X\right) ,$ where the operator $%
\chi _{a_{F}}:C\left( \left[ 0,a\right] ;X\right) \rightarrow C\left( \left[
0,a\right] ;X\right) $ is given by
\begin{equation*}
\chi _{a_{F}}\left( v\right) \left( t\right) =\left\{
\begin{array}{ll}
v\left( t\right) & \text{if }t\in \left[ 0,a_{F}\right] \\
v\left( a_{F}\right) & \text{if }t\in (a_{F},a].%
\end{array}%
\right.
\end{equation*}

We shall consider the integral operator $N:\overline{U}\rightarrow C\left( %
\left[ 0,a\right] ;X\right) $ defined by
\begin{equation}
N(u)(t)=T\left( t,0\right) BF\left( \int_{0}^{.}T(.,s)\Phi \left( u\right)
\left( s\right) ds\right) +\int_{0}^{t}T(t,s)\Phi \left( u\right) \left(
s\right) ds\,,\,\ t\in \lbrack 0,a].  \label{N}
\end{equation}%
Thus, any mild solution in $\overline{U}$ of (\ref{1}) is a fixed point of $%
N.$ Now M\"{o}nch's continuation theorem, Theorem \ref{th 1.1}, yields the
following very general existence principle for the problem (\ref{1}).

\begin{theorem}
\label{th 2.1}Assume that the conditions \emph{(h1)} and \emph{(h2)} hold.
In addition assume

\begin{description}
\item[(h3$^{0}$)] if $u=\lambda N\left( u\right) $ for some $u\in \overline{U%
}$ and $\lambda \in \left( 0,1\right) ,$ then $\left\vert u\left( t\right)
\right\vert <R\left( t\right) $ for all $t\in \left[ 0,a\right] .$

\item[(h4$^{0}$)] if $C\subset \overline{U}$ is countable and$\ C\subset
\overline{\text{\emph{conv}}}\left( \left\{ 0\right\} \cup N\left( C\right)
\right) ,$ then $\overline{C}$ is compact in $C\left( \left[ 0,a\right]
;X\right) .$
\end{description}

Then \emph{(\ref{1})} has a mild solution in $\overline{U}.$
\end{theorem}

To convert the general principle from Theorem 2.1 into applicable existence
criteria, we have to find sufficient conditions for (h3$^{0}$), (h4$^{0}$)
to hold. To this aim, we consider the operators $N_{1},N_{2}:\overline{U}%
\rightarrow C\left( \left[ 0,a\right] ;X\right) $ given by%
\begin{equation}
N_{1}\left( u\right) (t)=T\left( t,0\right) BF\left( \int_{0}^{.}T(.,s)\Phi
\left( u\right) \left( s\right) ds\right) ,  \label{N1}
\end{equation}%
\begin{equation}
N_{2}\left( u\right) \left( t\right) =\int_{0}^{t}T(t,s)\Phi \left( u\right)
\left( s\right) ds,  \label{N2}
\end{equation}%
for every $t\in \left[ 0,a\right] $ and $u\in \overline{U},$ and for
simplicity, we denote
\begin{equation*}
\left\vert BF\right\vert =\left\vert BF\right\vert _{\mathcal{L}(C\left(
[0,a];X\right) ,X)}.
\end{equation*}

\begin{lemma}
\label{lemma 2.2}Assume that the conditions \emph{(h1)} and \emph{(h2)}
hold. In addition assume that

\begin{description}
\item[(h3)] there exist $\delta \in L_{+}^{1}\left( 0,a\right) $ and a
continuous nondecreasing function $\psi :\mathbf{R}_{+}\rightarrow \mathbf{R}%
_{+}$ with $\psi \left( s\right) >0$ for all $s>0,$ such that%
\begin{equation}
\left\vert \Phi \left( u\right) \left( t\right) \right\vert \leq \delta
\left( t\right) \psi \left( \left\vert u\left( t\right) \right\vert \right)
\ \ \ \text{for a.a. }t\in \left[ 0,a\right] \text{ and all }u\in \overline{U%
},  \label{11}
\end{equation}%
\begin{equation}
r:=M^{2}\left\vert BF\right\vert \left\vert \delta \left( .\right) \psi
\left( R\left( .\right) \right) \right\vert _{L^{1}\left( 0,a_F\right)
}<\min_{t\in \left[ 0,a\right] }R\left( t\right) ,  \label{7}
\end{equation}%
where $\left[ 0,a_F\right] $ is the support of $F,$ and%
\begin{equation}
\int_{r}^{R\left( t\right) }\frac{d\tau }{\psi \left( \tau \right) }\geq
M\left\vert \delta \right\vert _{L^{1}\left( 0,t\right) }\ \ \ \text{for all
}t\in \left[ 0,a\right] ,  \label{8}
\end{equation}%
where $M$ is given by \emph{(\ref{M})}.
\end{description}

Then the condition \emph{(h3}$^{0}$\emph{)} is satisfied.
\end{lemma}

\begin{proof}
Let $u=\lambda N\left( u\right) $ for some $u\in \overline{U}$ $\ $and$\
\lambda \in \left( 0,1\right) .$ Then, for each $t\in \left[ 0,a\right] ,$
by (\ref{M}), (\ref{11}) and (\ref{7}), one has%
\begin{eqnarray}
\left\vert u\left( t\right) \right\vert &\leq &\lambda \left( \left\vert
N_{1}\left( u\right) (t)\right\vert +\left\vert N_{2}\left( u\right) \left(
t\right) \right\vert \right) \,  \notag  \label{c(t)} \\
&\leq &\lambda \left( M|BF\left( \chi _{a_F}(N_{2}(u))\right)
|+M\int_{0}^{t}|\Phi (u)(s)|\,ds\right) \,  \notag \\
&\leq &\lambda \left( M|BF||\chi
_{a_F}(N_{2}(u))|_{C([0,a];X)}+M\int_{0}^{t}\delta (s)\psi
(|u(s)|)\,ds\right) \,  \notag \\
&\leq &\lambda\left ( M^{2}|BF|\sup_{t\in \lbrack 0,a_F]}|\delta (.)\psi
(|u(.)|)|_{L^{1}(0,t)}+M|\delta (.)\psi (|u(.)|)|_{L^{1}(0,t)}\right)\,
\notag \\
&\leq &\lambda \left( r+M|\delta (.)\psi (|u(.)|)|_{L^{1}(0,t)}\right)
\,=:c(t).
\end{eqnarray}%
We show that
\begin{equation}
c(t)<R(t)\ \ \ \text{for every\ \ }t\in \lbrack 0,a].  \label{cR}
\end{equation}%
First we note that, by (\ref{7}), $c(0)<R(0).$ Then, suppose by
contradiction that there exists $t^{\ast }\in (0,a]$ such that $c(t^{\ast
})\geq R(t^{\ast });$ therefore, we may find an interval $\left[ 0,b\right]
\subset \left[ 0,a\right] $ with
\begin{equation*}
c\left( t\right) <R\left( t\right) \ \ \text{for every\ \ }t\in \lbrack
0,b),\ \ c\left( b\right) =R\left( b\right) .
\end{equation*}%
By using (\ref{c(t)}) and (h3), we have
\begin{equation*}
c^{\prime }\left( t\right) =\lambda M\delta \left( t\right) \psi \left(
\left\vert u\left( t\right) \right\vert \right) \leq \lambda M\delta \left(
t\right) \psi \left( c\left( t\right) \right) ,\ \ \text{for a.a. \ }t\in %
\left[ 0,b\right] .
\end{equation*}%
This implies
\begin{equation}
\int_{0}^{b}\frac{c^{\prime }\left( s\right) }{\psi \left( c\left( s\right)
\right) }\,ds\leq \lambda M\int_{0}^{b}\delta \left( s\right) \,ds.
\label{dis int}
\end{equation}%
Since $c\left( 0\right) =\lambda r\leq r,$ we have%
\begin{equation*}
\int_{0}^{b}\frac{c^{\prime }\left( s\right) }{\psi \left( c\left( s\right)
\right) }\,ds=\int_{c\left( 0\right) }^{c\left( b\right) }\frac{d\tau }{\psi
\left( \tau \right) }=\int_{\lambda r}^{R\left( b\right) }\frac{d\tau }{\psi
\left( \tau \right) }\geq \int_{r}^{R\left( b\right) }\frac{d\tau }{\psi
\left( \tau \right) };
\end{equation*}%
so by (\ref{dis int}) we deduce
\begin{equation*}
\int_{r}^{R\left( b\right) }\frac{d\tau }{\psi \left( \tau \right) }\leq
\lambda M|\delta |_{L^{1}(0,b)}.
\end{equation*}%
Then, if $\left\vert \delta \right\vert _{L^{1}\left( 0,b\right) }>0,$ we
obtain
\begin{equation*}
\int_{r}^{R\left( b\right) }\frac{d\tau }{\psi \left( \tau \right) }\leq
\lambda M\left\vert \delta \right\vert _{L^{1}\left( 0,b\right)
}<M\left\vert \delta \right\vert _{L^{1}\left( 0,b\right) },
\end{equation*}%
which contradicts (\ref{8}). Note that in our case $c\left( b\right)
=R\left( b\right) ,$ the equality $\left\vert \delta \right\vert
_{L^{1}\left( 0,b\right) }=0$ is not possible, since otherwise $c\left(
b\right) =\lambda r<R\left( b\right) ,$ which is impossible. Therefore $%
c\left( t\right) <R\left( t\right) $ for every $t\in \left[ 0,a\right] ,$
whence $\left\vert u\left( t\right) \right\vert <R\left( t\right) $ for all $%
t\in \left[ 0,a\right] ,$ as desired.
\end{proof}

\begin{remark}
\label{rem 1}\emph{In particular, if }$R\left( t\right) =R$\emph{\ (positive
constant) for every }$t\in \left[ 0,a\right] ,$\emph{\ than }$\overline{U}=%
\overline{B}_{C}(0,R)$ \emph{\ and the conditions (\ref{7}), (\ref{8}) read
as follows:}%
\begin{equation}
r:=M^{2}\left\vert BF\right\vert \psi \left( R\right) \left\vert \delta
\right\vert _{L^{1}\left( 0,a_{F}\right) }<R,  \label{7'}
\end{equation}%
\begin{equation}
\int_{r}^{R}\frac{d\tau }{\psi \left( \tau \right) }\geq M\left\vert \delta
\right\vert _{L^{1}\left( 0,a\right) }.  \label{8'}
\end{equation}
\end{remark}

In order to prove the next result, in correspondence to the function $R,$ we
introduce the undergraph of $2R,$
\begin{equation*}
V_{R}=\left\{ \left( t,s\right) \in \mathbf{R}^{2}\,:\,0\leq s\leq 2R\left(
t\right) ,\ \,0\leq t\leq a\right\}
\end{equation*}%
and we say that a function $\omega :V_{R}\rightarrow \mathbf{R}_{+}$ is $%
L^{1}$-\textit{Carath\'{e}odory on the undergraph }$V_{R}$ if

\begin{description}
\item[$\left( \protect\omega 1\right) $] $\omega \left( .,s\right) $ is
measurable on $\{t\in \lbrack 0,a]\,:\,2R\left( t\right) \geq s\}$ for every
$s\in \left[ 0,2\left\vert R\right\vert _{L^{\infty }\left( 0,a\right) }%
\right] ;$

\item[$\left( \protect\omega 2\right) $] $\omega \left( t,.\right) $ is
continuous on $\left[ 0,2R\left( t\right) \right] ,$ for a.a. $t\in \lbrack
0,a];$

\item[$\left( \protect\omega 3\right) $] there exists $\eta \in
L_{+}^{1}\left( 0,a\right) $ such that $\omega \left( t,s\right) \leq \eta
(t),$ for all $s\in \left[ 0,2R\left( t\right) \right] $ and a.a. $t\in %
\left[ 0,a\right] .$
\end{description}

Moreover, we shall assume the following property

\begin{description}
\item[(h4)] there exists a function $\omega :V_{R}\rightarrow \mathbf{R}_{+}$
which is $L^{1}$-Carath\'{e}odory on the undergraph $V_{R}$ and such that
for each countable set $C\subset \overline{U},$%
\begin{equation}
\alpha \left( \Phi \left( C\right) \left( t\right) \right) \leq \omega
\left( t,\alpha \left( C\left( t\right) \right) \right) ,\ \ \text{for a.a. }%
t\in \left[ 0,a\right]  \label{2}
\end{equation}%
and that the unique solution $\varphi \in C\left[ 0,a\right] $ with graph$%
\left( \varphi \right) \subset V_{R}$ of the inequality%
\begin{eqnarray}
\varphi \left( t\right) &\leq &2M^{2}\left\vert BF\right\vert
\int_{0}^{a_F}\omega \left( s,\varphi \left( s\right) \right) \,ds\,
\label{6} \\
&&+\,2M\int_{0}^{t}\omega \left( s,\varphi \left( s\right) \right) \,ds,\ \
\text{for all}\ \ t\in \left[ 0,a\right]  \notag
\end{eqnarray}%
is $\varphi \equiv 0.$
\end{description}

Note that the condition (h4) is well posed; indeed, if $C\subset \overline{U}%
,$ then
\begin{equation*}
\alpha (C(t))\leq \alpha (B(0,R(t)))=2R(t),\ \ \text{for all}\ \ t\in
\lbrack 0,a].
\end{equation*}%
Hence, $(t,\alpha (C(t)))\in V_{R},$ for every $t\in \lbrack 0,a].$

\begin{remark}[the Kamke function of a nonlocal problem]
\emph{In the case of the classical Cauchy problem, when }$A\left( t\right)
=0 $\emph{\ for every }$t\in \left[ 0,a\right] $\emph{\ and }$F=0$\emph{\
(equivalently, when }$a_F=0$\emph{$)$, the inequality (\ref{6}) reduces to}%
\begin{equation*}
\varphi \left( t\right) \leq 2\int_{0}^{t}\omega \left( s,\varphi \left(
s\right) \right) ds,\ \ \ \text{for all\ }\ t\in \left[ 0,a\right]
\end{equation*}%
\emph{and the condition required in (h4) means that }$\omega $\emph{\ is a }%
Kamke function of the Cauchy problem.\emph{\ }\noindent \emph{By analogy, in
the case of our nonlocal problem (\ref{1}), the function }$\omega $\emph{\
in (h4) can be called a} Kamke function of the nonlocal initial value
problem.
\end{remark}

\begin{lemma}
\label{II lemma} Assume the conditions \emph{(h1)}, \emph{(h2)}, \emph{(h4)}
and

\begin{description}
\item[(h3')] there exist $\delta \in L_{+}^{1}\left( 0,a\right) $ and a
continuous nondecreasing function $\psi :\mathbf{R}_{+}\rightarrow \mathbf{R}%
_{+}$ with $\psi \left( s\right) >0$ for all $s>0,$ such that \emph{(\ref{11}%
)} holds.
\end{description}

Then the condition \emph{(h4}$^{0}$\emph{)} is satisfied.
\end{lemma}

\noindent \textbf{Proof.} Let $C\subset \overline{U}$ be countable with
\begin{equation}
C\subset \overline{\text{conv}}\left( \left\{ 0\right\} \cup N\left(
C\right) \right) ,  \label{Csub}
\end{equation}%
where $N$ is given by (\ref{N}). \smallskip First we show that $C$ is
equicontinuous. For this, it is enough to prove the equicontinuity of the
set $N\left( C\right) .$ \noindent First of all, we have that $N_{2}\left(
C\right) $ is equicontinuous. In fact, for any fixed $\varepsilon >0,$ in
correspondence to $\varepsilon /6M,$ there exists $\eta (\varepsilon /6M)>0$
such that for every measurable set $\mathcal{M}$ with $\lambda (\mathcal{M}%
)<\eta (\varepsilon /6M)$ (where $\lambda $ denotes the Lebesgue measure on $%
[0,a]$) one has $\int_{\mathcal{M}}\delta (s)\psi (R(s))\,ds<\varepsilon
/6M, $ where $\delta (.)\psi (R(.))\in L_{+}^{1}(0,a)$ (see (h3')). Let us
fix $\gamma >0$ with $\gamma <\eta (\varepsilon /6M).$ For any $u\in C$ and $%
t,\overline{t}\in \lbrack 0,a]$ with $0<t-\overline{t}<\gamma ,$ by using (%
\ref{M}) and hypothesis (h3'), we have%
\begin{eqnarray}
&&\left\vert N_{2}\left( u\right) \left( t\right) -N_{2}\left( u\right)
\left( \overline{t}\right) \right\vert  \label{E2} \\
&=&\left\vert \int_{0}^{\overline{t}}T(t,s)\Phi (u)(s)\,ds+\int_{\overline{t}%
}^{t}T(t,s)\Phi (u)(s)\,ds-\int_{0}^{\overline{t}}T(\overline{t},s)\Phi
(u)(s)\,ds\right\vert \,  \notag \\
&\leq &\int_{0}^{\overline{t}}\left\vert T(t,s)-T(\overline{t},s)\right\vert
_{\mathcal{L}(X,X)}|\Phi (u)(s)|\,ds+M\int_{\overline{t}}^{t}|\Phi
(u)(s)|\,ds\,  \notag \\
&\leq &\int_{0}^{\overline{t}}\left\vert T(t,s)-T(\overline{t},s)\right\vert
_{\mathcal{L}(X,X)}\delta (s)\psi (R(s))\,ds+M\int_{\overline{t}}^{t}\delta
(s)\psi (R(s))\,ds  \notag \\
&\leq &\int_{0}^{\overline{t}-\gamma }\left\vert T(t,s)-T(\overline{t}%
,s)\right\vert _{\mathcal{L}(X,X)}\delta (s)\psi (R(s))\,ds+\,2M\int_{%
\overline{t}-\gamma }^{\overline{t}}\delta (s)\psi (R(s))\,ds  \notag \\
&&+M\int_{\overline{t}}^{t}\delta (s)\psi (R(s))\,ds\,  \notag \\
&\leq &\int_{0}^{\overline{t}-\gamma }\left\vert T(t,s)-T(\overline{t}%
,s)\right\vert _{\mathcal{L}(X,X)}\delta (s)\psi (R(s))\,ds+\varepsilon
/3+\varepsilon /6.  \notag
\end{eqnarray}%
\noindent Let $H:=\int_{0}^{a}\delta (s)\psi (R(s))\,ds.$ By the uniform
continuity of the evolution operator $T,$ there exists $\eta (\varepsilon
/3H)>0$ which can be chosen with $\eta (\varepsilon /3H)\leq \gamma ,$ such
that if $0<t-\overline{t}<\eta (\varepsilon /3H)\,,\,s\in \lbrack 0,%
\overline{t}],$ then $|T(t,s)-T(\overline{t},s)|_{\mathcal{L}%
(X,X)}<\varepsilon /3H.$ So (\ref{E2}) yields
\begin{equation*}
\left\vert N_{2}(u)(t)-N_{2}(u)(\overline{t})\right\vert \leq \varepsilon
/3\,+\,\varepsilon /3\,+\,\varepsilon /6\,<\varepsilon .
\end{equation*}%
Hence $N_{2}(C)$ is equicontinuous.

To prove that $N_{1}\left( C\right) $ is equicontinuous, first observe that
by (\ref{N2}), the map $N_{1}$ in (\ref{N1}) can be written as
\begin{equation*}
N_{1}\left( u\right) (t)=T\left( t,0\right) BF\left( N_{2}(u)\right) ,
\end{equation*}%
for all $t\in \lbrack 0,a]$ and $u\in \overline{U}.$ Denote
\begin{equation}
\widetilde{M}=M\left\vert BF\right\vert |\delta (.)\psi (R(.))|_{L^{1}(0,a)}.
\label{tildeM}
\end{equation}%
By the continuity of the evolution operator $T,$ we have that for every $%
\varepsilon >0,$ there exists $\eta (\varepsilon /\widetilde{M})>0$ such
that for every $t,\overline{t}\in \lbrack 0,a]$ with $|t-\overline{t}|<\eta
(\varepsilon /\widetilde{M}),$ we have
\begin{equation}
|T(t,0)-T(\bar{t},0)|_{\mathcal{L}(X,X)}<\varepsilon /\widetilde{M}.
\label{T0}
\end{equation}%
Assuming without less of generality that $t>\overline{t},$ according to (\ref%
{N2}), (h3'), (\ref{tildeM}) and (\ref{T0}), for every $u\in C,$ we have the
following estimation%
\begin{eqnarray*}
&&|N_{1}(u)(t)-N_{1}(u)(\overline{t})|\, \\
&=&\,|[T(t,0)-T(\overline{t},0)]BF(N_{2}(u))|\, \\
&\leq &|T(t,0)-T(\overline{t},0)|_{\mathcal{L}%
(X,X)}|BF||N_{2}(u)|_{C([0,a];X) }\, \\
&\leq &|T(t,0)-T(\overline{t},0)|_{\mathcal{L}(X,X)}M|BF|\sup_{t\in \lbrack
0,a]}\int_{0}^{t}|\Phi (u)(s)|\,ds\, \\
&\leq &|T(t,0)-T(\overline{t},0)|_{\mathcal{L}(X,X)}M|BF|\int_{0}^{a}\delta
(s)\psi (R(s))\,ds\, \\
&=&\,|T(t,0)-T(\overline{t},0)|_{\mathcal{L}(X,X)}\widetilde{M}%
\,<\,\varepsilon .
\end{eqnarray*}%
So $N_{1}(C)$ is equicontinuous. \noindent Hence, by (\ref{N}), (\ref{N1})
and (\ref{N2}), we have the equicontinuity of $N(C).$ Therefore, by (\ref%
{Csub}), the set $C$ is equicontinuous too. \noindent Furthermore, for every
fixed $t\in \lbrack 0,a],$ the set $C(t)$ is relatively compact in $X.$
Indeed, $C$ is bounded in $C([0,a];X)$ and
\begin{eqnarray}
\alpha \left( C\left( t\right) \right) &\leq &\alpha \left( \overline{\text{%
conv}}\left( \left\{ 0\right\} \cup N\left( C\right) \left( t\right) \right)
\right) =\alpha \left( N\left( C\right) \left( t\right) \right)  \label{3} \\
&\leq &\alpha \left( N_{1}\left( C\right) (t)\right) +\alpha \left(
N_{2}\left( C\right) \left( t\right) \right) .  \notag
\end{eqnarray}%
According to (\ref{h1}) and (h4), we have
\begin{eqnarray}
\alpha (N_{2}(C)(t)) &\leq &2\int_{0}^{t}\alpha (T(t,s)\Phi (C)(s))\,ds\,
\label{4} \\
&\leq &2M\int_{0}^{t}\omega (s,\alpha (C(s)))\,ds.  \notag
\end{eqnarray}

\noindent In addition, using the linearity of the mapping $BF$ and (\ref{a1}%
), we deduce that%
\begin{eqnarray}
\alpha \left( N_{1}(C)(t)\right) &\leq &M|BF|\alpha _{C}\left( \chi
_{a_F}\left( N_{2}(C)\right) \right) \,  \label{5} \\
&=&M|BF|\max_{t\in \lbrack 0,a_F]}\alpha (N_{2}(C)(t))\,  \notag \\
&\leq &2M^{2}\left\vert BF\right\vert \int_{0}^{a_F}\omega \left( s,\alpha
\left( C\left( s\right) \right) \right) \,ds.  \notag
\end{eqnarray}%
Now (\ref{3}), (\ref{4}) and (\ref{5}) give%
\begin{equation*}
\alpha \left( C\left( t\right) \right) \leq 2M^{2}\left\vert BF\right\vert
\int_{0}^{a_F}\omega \left( s,\alpha \left( C\left( s\right) \right) \right)
\,ds+2M\int_{0}^{t}\omega (s,\alpha (C(s)))\,ds.
\end{equation*}%
Hence the function
\begin{equation*}
\varphi \left( t\right) =\alpha \left( C\left( t\right) \right) ,\ \ \ \text{%
for all \ }t\in \left[ 0,a\right]
\end{equation*}%
solves (\ref{6}). In addition $\varphi $ is continuous on $\left[ 0,a\right]
$ and its graph is contained in $V.$ Consequently, $\varphi \equiv 0,$ that
is $\alpha \left( C\left( t\right) \right) =0$ for all $t\in \left[ 0,a%
\right] .$ Thus $C\left( t\right) $ is relatively compact in $X$ for each $%
t\in \left[ 0,a\right] ,$ as desired. \hfill $\Box $\medskip

Now Theorem \ref{th 2.1} and Lemmas \ref{lemma 2.2} and \ref{II lemma} yield
the main existence result for (\ref{1}).

\begin{theorem}
\label{thm 2.2}Assume that the conditions \emph{(h1)-(h4)} are satisfied.
Then \emph{(\ref{1})} has a mild solution in $\overline{U}.$
\end{theorem}

In the setting of Remark \ref{rem 1}, from Theorem \ref{thm 2.2} we deduce
the following result.

\begin{corollary}[case of time-independent radius]
Assume that the conditions \emph{(h1)}, \emph{(h2)} and \emph{(h4)} hold,
where $\overline{U}=\overline{B}_{C}(0,R),$ $R>0$ and $V_{R}=\{\left(
t,s\right) \in \mathbf{R}^{2}\,:\,0\leq s\leq 2R\,,$ $0\leq t\leq a\}.$ In
addition assume that

\begin{description}
\item[(h3*)] there exist $\delta \in L_{+}^{1}\left( 0,a\right) $ and a
continuous nondecreasing map function $\psi :\mathbf{R}_{+}\rightarrow
\mathbf{R}_{+}$ with $\psi \left( s\right) >0$ for all $s>0,$ such that
\begin{equation*}
\left\vert \Phi (u)(t)\right\vert \leq \delta (t)\psi \left( \left\vert
u(t)\right\vert \right) ,\ \ \text{for a.a. }t\in \left[ 0,a\right] \ \
\text{and all}\ \ u\in \overline{B}_{C}(0,R),\,
\end{equation*}%
\begin{equation}
\frac{R}{\psi \left( R\right) }\geq M^{2}\left\vert BF\right\vert \left\vert
\delta \right\vert _{L^{1}\left( 0,a_F\right) }+M\left\vert \delta
\right\vert _{L^{1}\left( 0,a\right) },  \label{9}
\end{equation}%
where $M$ is from \emph{(\ref{M})}.
\end{description}

Then \emph{(\ref{1})} has a mild solution in $\overline{B}_{C}(0,R).$
\end{corollary}

\noindent \textbf{Proof.} First of all, we show that under conditions (h1),
(h2) and (h3$^{\ast }$), the condition (h3) is satisfied in the case of
time-independent radius. It is easy to see that (\ref{7'}) follows from (\ref%
{9}) if $\left\vert \delta \right\vert _{L^{1}\left( 0,a\right) }>0;$
otherwise (\ref{7'}) is trivially satisfied. \noindent Furthermore, since
the function $\psi $ is nondecreasing, we have
\begin{equation*}
\int_{r}^{R}\frac{d\tau }{\psi \left( \tau \right) }\geq \frac{R-r}{\psi
\left( R\right) }
\end{equation*}%
and thus, by (\ref{9}) and the definition of $r$ (see (\ref{7'})), condition
(\ref{8'}) holds. According to Lemma \ref{lemma 2.2} and Remark \ref{rem 1},
in the case of time-independent radius, the condition (h3$^{0}$) is
satisfied. Now Theorem \ref{thm 2.2} finishes the proof. \hfill $\Box $
\medskip

Note that the condition (\ref{9}) guarantees even more, namely that $N\left(
\overline{U}\right) \subset \overline{U}.$

A much more applicable result can be derived from Theorem \ref{thm 2.2}.

\begin{theorem}
\label{thm2.3}Assume that \emph{(h1), (h2)} and \emph{(h3)} hold. In
addition assume that the following condition is satisfied:

\begin{description}
\item[(h4*)] $\Phi =\Psi +\Theta ,$ where $\Theta \left( \overline{U}\right)
\left( t\right) \subset K$ for a.a $t\in \left[ 0,a\right] ,$ $K$ being a
compact set in $X,$ and there exists $\gamma \in L_{+}^{1}\left( 0,a\right) $
such that for each countable set $C\subset \overline{U},$
\begin{equation}
\alpha \left( \Psi \left( C\right) \left( t\right) \right) \leq \gamma
\left( t\right) \alpha \left( C\left( t\right) \right) ,\ \ \ \text{for a.a.
~}t\in \left[ 0,a\right]  \label{10}
\end{equation}%
and%
\begin{equation}
\left( 2M^{2}\left\vert BF\right\vert +2M\right) \left\vert \gamma
\right\vert _{L^{1}\left( 0,a_F\right) }<1.  \label{13}
\end{equation}
\end{description}

Then \emph{(\ref{1})} has a mild solution in $\overline{U}.$
\end{theorem}

\begin{proof}
We shall check (h4). Since $\Theta \left( \overline{U}\right) \left(
t\right) \subset K$ for a.a $t\in \left[ 0,a\right] ,$ $K$ being a compact
set in $X,$ from (\ref{10}) we see that (\ref{2}) holds with $\omega \left(
t,s\right) =\gamma \left( t\right) s,\ \left( t,s\right) \in V_{R}.$ Now let
$\varphi \in C\left[ 0,a\right] $ with graph$\left( \varphi \right) \subset
V_{R},$ be any solution of (\ref{6}), that is%
\begin{equation}
\varphi \left( t\right) \leq 2M^{2}\left\vert BF\right\vert \left\vert
\gamma \varphi \right\vert _{L^{1}\left( 0,a_{F}\right) }+2M\left\vert
\gamma \varphi \right\vert _{L^{1}\left( 0,t\right) },\ \ t\in \left[ 0,a%
\right] .  \label{i1}
\end{equation}%
First we show that $\varphi \left( t\right) =0$ for all $t\in \left[ 0,a_{F}%
\right] .$ Indeed, from (\ref{i1}), since $\varphi $ is nonnegative, we
deduce%
\begin{equation}
\left\vert \varphi \right\vert _{L^{\infty }\left( 0,a_{F}\right) }\leq
\left\vert \varphi \right\vert _{L^{\infty }\left( 0,a_{F}\right) }\left(
2M^{2}\left\vert BF\right\vert +2M\right) \left\vert \gamma \right\vert
_{L^{1}\left( 0,a_{F}\right) },  \label{mf}
\end{equation}%
which in view of (\ref{13}) gives $\left\vert \varphi \right\vert
_{L^{\infty }\left( 0,a_{F}\right) }=0.$ Then from the continuity of $%
\varphi ,$ we deduce $\varphi \left( t\right) =0$ for all $t\in \left[
0,a_{F}\right] ,$ as claimed. As a consequence, (\ref{i1}) reduces to
\begin{equation*}
\varphi \left( t\right) \leq 2M\int_{a_{F}}^{t}\gamma \left( s\right)
\varphi \left( s\right) ds,\ \ \ \text{for all\ }\ t\in \left[ a_{F},a\right]
,
\end{equation*}%
and the remaining conclusion $\varphi \left( t\right) =0$ for $t\in
(a_{F},a] $ follows from Gronwall's inequality. Then \textbf{(h4)} holds. By
Theorem \ref{thm 2.2} the thesis is reached.
\end{proof}

\begin{remark}
\emph{In particular, the condition (\ref{10}) holds if }$\Psi $\emph{\
satisfies the Lipschitz inequality}%
\begin{equation*}
\left\vert \Psi \left( u\right) \left( t\right) -\Psi \left( v\right) \left(
t\right) \right\vert \leq \gamma \left( t\right) \left\vert u\left( t\right)
-v\left( t\right) \right\vert
\end{equation*}%
\emph{for all }$u,v\in \overline{U}$\emph{\ and a.a. }$t\in \left[ 0,a\right]
.$
\end{remark}

In the case of the superposition nonlinearity, namely if $\Phi $\ is given
by (\ref{so}), from Theorem \ref{thm2.3}, we can deduce the following result.

\begin{corollary}[case of superposition operator]
Assume that the condition \emph{(h2)} holds. Let $f:[0,a]\times \overline{B}%
(0,|R|_{\infty })\rightarrow X$ be a mapping such that

\begin{description}
\item[(h1$_{f}$)] $\ f\left( .,x\right) $\ is measurable on $[0,a]$ for each
$x\in \overline{B}(0,|R|_{\infty });\newline
$\newline
\ \qquad $f\left( t,.\right) $\ is continuous on the ball $\overline{B}%
(0,R(t))$ for a.a. $t\in \left[ 0,a\right] ;\newline
$\newline
\ \qquad $\left\vert f\left( t,x\right) \right\vert \leq \eta \left(
t\right) $ for all $x\in \overline{B}(0,R(t))$ and a.a. $t\in \left[ 0,a%
\right] ,$\ where $\eta \in L_{+}^{1}\left( 0,a\right) ;$
\end{description}

\begin{description}
\item[(h3$_{f}$)] \ there exist $\delta \in L_{+}^{1}\left( 0,a\right) $ and
a continuous nondecreasing function $\psi :\mathbf{R}_{+}\rightarrow \mathbf{%
R}_{+}$ with $\psi \left( s\right) >0$ for all $s>0,$ such that
\begin{equation*}
\left\vert f(t,x)\right\vert \leq \delta (t)\psi \left( \left\vert
x)\right\vert \right) ,\ \ \text{for a.a. \ }t\in \left[ 0,a\right] \ \
\text{and all}\ \ x\in \overline{B}(0,R(t)),
\end{equation*}%
and \emph{(\ref{7}), (\ref{8})} are satisfied;
\end{description}

\begin{description}
\item[(h4*$_{f}$)] $\ f=g+h,$ where $h(D)$ is relatively compact in $X$ for $%
D:=\{(t,x)\,:\,|x|\leq R(t),\,t\in \lbrack 0,a]\},$ and there exists $\gamma
\in L_{+}^{1}\left( 0,a\right) $ such that for each countable set $C\subset
\overline{B}(0,R(t)),$
\begin{equation*}
\alpha \left( g\left( t,C\right) \right) \leq \gamma \left( t\right) \alpha
\left( C\right) ,\ \ \text{for a.a. }\ t\in \lbrack 0,a]
\end{equation*}%
and \emph{(\ref{13})} holds.
\end{description}

Then the problem
\begin{equation*}
\left\{
\begin{array}{l}
u^{\prime }(t)=A(t)u(t)+f(t,u(t)),\ \ \ \text{for\ a.a.\ }t\in \left[ 0,a%
\right] \\
u(0)=F(u)%
\end{array}%
\right.
\end{equation*}%
has a mild solution in $\overline{U}.$
\end{corollary}

\section{Existence and localization of solutions for evolution systems}

\setcounter{equation}{0}Consider $n$ Banach spaces $\left( X_{i},\left\vert
.\right\vert _{i}\right) ,$ the product space $X=X_{1}\times X_{2}\times
...\times X_{n}$ and the Cauchy problem for an $n$-dimensional system, with
nonlocal conditions%
\begin{equation}
\left\{
\begin{array}{l}
u_{i}^{\prime }\left( t\right) =A_{i}\left( t\right) u_{i}\left( t\right)
+\Phi _{i}\left( u_{1},u_{2},...,u_{n}\right) \left( t\right) ,\ \ \ \text{%
for a.a\ \ }t\in \left[ 0,a\right] \\
u_{i}\left( 0\right) =F_{i}\left( u_{1},u_{2},...,u_{n}\right) ,%
\end{array}%
\right.  \label{s1}
\end{equation}%
$i=1,2,...,n.$ Here, for each $i,$ $\left\{ A_{i}\left( t\right) \right\}
_{t\in \left[ 0,a\right] }$ is a family of linear operators in the Banach
space $X_{i}$ generating an evolution operator $T_{i},$ $\ \Phi _{i}$ is a
nonlinear mapping, and $F_{i}$ is linear.

\noindent

On the linear part of the $i$-equation we require the condition:

\begin{description}
\item[(A$_{i}$)] $\{A_{i}(t)\}_{t\in \lbrack 0,a]}$ is a family of linear
not necessarily bounded operators $(A_{i}(t):D(A_{i})\subset
X_{i}\rightarrow X_{i},$ $t\in \lbrack 0,a],$ $D(A_{i})$ is a dense subset
of $X_{i}$ not depending on $t)$ generating a continuous evolution operator $%
T_{i}:\Delta \rightarrow \mathcal{L}(X_{i},X_{i}).$
\end{description}

\medskip Consider the vector-valued mappings, represented as column
matrices, $\Phi $ and $F$ acting from $C\left( \left[ 0,a\right] ;X\right) $
into $X,$%
\begin{equation}
\Phi =\left[ \Phi _{1},\Phi _{2},...,\Phi _{n}\right] ^{\text{tr}},\ \ \ F=%
\left[ F_{1},F_{2},...,F_{n}\right] ^{\text{tr}}  \label{fif}
\end{equation}%
and the family $\left\{ A\left( t\right) \right\} _{t\in \left[ 0,a\right] }$
of linear operators in $X,$ where, for each $t\in \lbrack 0,a]$, the
opertator $A(t):D(A)=\prod_{i=1}^{n}D(A_{i})\rightarrow X$ is represented as
diagonal matrix of operators,%
\begin{equation*}
A\left( t\right) =\left[
\begin{array}{ccc}
A_{1}\left( t\right) & ... & 0 \\
... & A_{2}\left( t\right) & ... \\
0 & ... & A_{n}\left( t\right)%
\end{array}%
\right] .
\end{equation*}%
Clearly, $\ A(t)x=[A_{1}(t)x_{1},A_{2}(t)x_{2},...,A_{n}(t)x_{n}]^{\text{tr}%
},$ $x\in D(A).$ Then looking at the elements of the product space $X$ as
column matrices, the system (\ref{s1}) can be written as
\begin{equation*}
\left\{
\begin{array}{l}
u^{\prime }\left( t\right) =A\left( t\right) u\left( t\right) +\Phi \left(
u\right) \left( t\right) ,\ \ \ \text{for\ a.a.\ }t\in \left[ 0,a\right] \\
u\left( 0\right) =F\left( u\right) ,%
\end{array}%
\right.
\end{equation*}%
which is exactly problem (\ref{1}), this time, in a vectorial form, in the
product space $X=X_{1}\times X_{2}\times ...\times X_{n}.$ Thus, all
previous results are applicable and yield existence theorems for the system (%
\ref{s1}). However, like in \cite{bip}, we can take advantage from the
splitting of this vectorial equation into $n$ equations and obtain more
refined results under conditions allowing the operators $F_{i}$ and $\Phi
_{i}$ to behave independently as much as possible. This will be possible by
exploiting the vectorial nature of the system and by using matrix conditions
instead of scalar ones. For instance, instead of speaking globally about the
support of the operator $F,$ as shown by (\ref{supp}), we shall consider the
\textit{support of }$F$\textit{\ with respect to each variable} $u_{i},$ $%
i=1,2,...,n,$ as being the minimal closed subinterval $\left[ 0,a_{i}\right]
$ of $\left[ 0,a\right] $ with the property%
\begin{eqnarray*}
F\left( u_{1},...,u_{i-1},u_{i},u_{i+1},...,u_{n}\right) &=&F\left(
u_{1},...,u_{i-1},v_{i},u_{i+1},...,u_{n}\right) \text{ } \\
\text{whenever }u_{i} &=&v_{i}\text{ on }\left[ 0,a_{i}\right] .
\end{eqnarray*}

Also, we are interested not only on the existence of a mild solution $%
u=\left( u_{1},u_{2},...,u_{n}\right) $ of the problem (\ref{s1}), but also
on the localization of each component $u_{i}$ individually. Thus, the
solutions are sought in a bounded closed subset $\overline{U}$ of $C\left( %
\left[ 0,a\right] ;X\right) ,$ of the form $\overline{U}=\overline{U}%
_{1}\times \overline{U}_{2}\times ...\times \overline{U}_{n}$ with%
\begin{equation*}
\overline{U}_{i}:=\left\{ v\in C\left( \left[ 0,a\right] ;X_{i}\right)
:\left\vert v\left( t\right) \right\vert _{i}\leq R_{i}\left( t\right) \text{
for all }t\in \left[ 0,a\right] \right\} ,
\end{equation*}%
where $R_{i}\in C\left[ 0,a\right] $ are given functions with $R_{i}\left(
t\right) >0\ $for all $t\in \left[ 0,a\right] ,$ $i=1,2,...,n.$

Let us define the family $\left\{ T\left( t,s\right) \right\} _{(t,s)\in
\Delta }$ of linear operators from $X$ to $X,$ where, for each $(t,s)\in
\Delta ,$ $T(t,s)$ is represented as diagonal matrix
\begin{equation*}
T\left( t,s\right) =\left[
\begin{array}{ccc}
T_{1}\left( t,s\right) & ... & 0 \\
... & T_{2}\left( t,s\right) & ... \\
0 & ... & T_{n}\left( t,s\right)%
\end{array}%
\right] ,
\end{equation*}%
and so%
\begin{equation*}
T\left( t,s\right) x=\left[ T_{1}\left( t,s\right) x_{1},T_{2}\left(
t,s\right) x_{2},...,T_{n}\left( t,s\right) x_{n}\right] ^{\text{tr}}\ ,\ \
x\in X.
\end{equation*}

We shall assume the analogue conditions to (h1) and (h2):

\begin{description}
\item[(H1)] $\Phi _{i}:\overline{U}\rightarrow L^{1}(0,a;X_{i})$ is
continuous, $i=1,2,...,n;$

\item[(H2)] $F_{i}:C\left( \left[ 0,a\right] ;X\right) \rightarrow X_{i}$ is
a linear and continuous mapping, $i=1,2,...,n,$ and the operator from $X$ to
$X,$ $x\mapsto x-F\left( T\left( .,0\right) x\right) $ has an inverse $B.$
\end{description}

Note that, using the vectorial notations $\Phi $ and $F$ given in (\ref{fif}%
), the conditions (H1), (H2) appear identical to (h1), (h2), respectively.

Like $F,$ the linear operator $B$ from $X$ to $X$ can be naturaly looked as
a column matrix
\begin{equation*}
B=\left[ B_{1},B_{2},...,B_{n}\right] ^{\text{tr}},
\end{equation*}%
where $B_{i}\in \mathcal{L}\left( X,X_{i}\right) .$ Moreover, thanks to the
linearity of the operators $B_{i}$ and $F_{i},$ $B$ and $F$ can be
identified to a matrix
\begin{equation*}
B=\left[ B_{ij}\right] _{1\leq i,j\leq n},\ \ \ F=\left[ F_{ij}\right]
_{1\leq i,j\leq n},
\end{equation*}%
whose entries $B_{ij}\in \mathcal{L}\left( X_{j},X_{i}\right) ,$ $F_{ij}\in $
$\mathcal{L}\left( C\left( \left[ 0,a\right] ;X_{j}\right) ,X_{i}\right) $
are given by%
\begin{eqnarray*}
B_{ij}\left( x_{j}\right) &=&B_{i}\left( 0,0,...,x_{j},0,...,0\right) \\
F_{ij}\left( u_{j}\right) &=&F_{i}\left( 0,0,...,u_{j},0,...,0\right) ,
\end{eqnarray*}%
with $x_{j}\in X_{j},$ $u_{j}\in C\left( \left[ 0,a\right] ;X_{j}\right) $
on the $j$-th position. Then%
\begin{equation*}
B_{i}\left( x\right) =\sum_{j=1}^{n}B_{ij}\left( x_{j}\right) ,\ \ \ \text{%
for every\ \ }x\in X,
\end{equation*}%
\begin{equation*}
F_{i}\left( u\right) =\sum_{j=1}^{n}F_{ij}\left( u_{j}\right) ,\ \ \ \text{%
for every \ }u\in C\left( \left[ 0,a\right] ;X\right) .
\end{equation*}

Let $G$ denote the linear mapping $BF$ from $C\left( \left[ 0,a\right]
;X\right) $ to $X.$ According to the above explanations,
\begin{equation*}
G\left( u\right) =\left[ G_{1}\left( u\right) ,G_{2}\left( u\right)
,...,G_{n}\left( u\right) \right] ^{\text{tr}},\ \ \ G=\left[ G_{ij}\right]
_{1\leq i,j\leq n},
\end{equation*}%
where $G_{i}\in \mathcal{L}\left( C\left( \left[ 0,a\right] ;X\right)
,X_{i}\right) ,$ $G_{ij}\in \mathcal{L}\left( C\left( \left[ 0,a\right]
;X_{j}\right) ,X_{i}\right) $ and
\begin{equation*}
G_{ij}\left( u_{j}\right) =G_{i}\left( 0,0,..,u_{j},0,...,0\right)
\end{equation*}%
with $u_{j}$ on the $j$-th position. Thanks again to the linearity of the
operators, we have
\begin{eqnarray*}
G_{i}\left( u\right) &=&B_{i}\left( F\left( u\right) \right)
=\sum_{k=1}^{n}B_{ik}\left( F_{k}\left( u\right) \right)
=\sum_{k=1}^{n}B_{ik}\left( \sum_{j=1}^{n}F_{kj}\left( u_{j}\right) \right)
\\
&=&\sum_{k,j=1}^{n}B_{ik}\left( F_{kj}\left( u_{j}\right) \right)
\end{eqnarray*}%
and%
\begin{equation*}
G_{ij}\left( u_{j}\right) =\sum_{k=1}^{n}B_{ik}F_{kj}\left( u_{j}\right) .
\end{equation*}

Using the above notations, letting $M_{i}$ be such that $|T_{i}(t,s)|_{%
\mathcal{L}(X_{i},X_{i})}\leq M_{i}\,$ for all $(t,s)\in \Delta ,$ and
denoting for simplicity%
\begin{equation*}
\left\vert G_{ij}\right\vert =\left\vert G_{ij}\right\vert _{\mathcal{L}%
(C([0,a];X_{j}),X_{i})},
\end{equation*}%
we can state our next assumption:

\begin{description}
\item[(H3)] for each $i=1,2,...,n,$ there exist $\delta _{i}\in
L_{+}^{1}\left( 0,a\right) $ and a continuous nondecreasing function $\psi
_{i}:\mathbf{R}_{+}\rightarrow \mathbf{R}_{+}$ with $\psi _{i}\left(
s\right) >0$ for all $s>0,$ such that%
\begin{equation}
\left\vert \Phi _{i}\left( u\right) \left( t\right) \right\vert \leq \delta
_{i}\left( t\right) \psi _{i}\left( \left\vert u_{i}\left( t\right)
\right\vert _{i}\right) \ \ \ \text{for a.a. }t\in \left[ 0,a\right] \text{
and all }u\in \overline{U},  \label{ff1}
\end{equation}%
\begin{equation}
r_{i}:=M_{i}\sum_{j=1}^{n}\left\vert G_{ij}\right\vert M_{j}\left\vert
\delta _{j}\left( .\right) \psi _{j}\left( R_{j}\left( .\right) \right)
\right\vert _{L^{1}\left( 0,a_{j}\right) }<\min_{t\in \left[ 0,a\right]
}R_{i}\left( t\right) ,  \label{rr1}
\end{equation}%
where $\left[ 0,a_{j}\right] $ is the support of $F$ with respect the
variable $u_{j},$ and%
\begin{equation}
\int_{r_{i}}^{R_{i}\left( t\right) }\frac{d\tau }{\psi _{i}\left( \tau
\right) }\geq M_{i}\left\vert \delta _{i}\right\vert _{L^{1}\left(
0,t\right) }\ \ \ \text{for all }t\in \left[ 0,a\right] .
\end{equation}
\end{description}

Note that the support of $F$ in this case is given by $a_{F}=\max_{1\leq
i\leq n}a_{i}.$

Finally, if we denote by $\alpha _{i}$ the Kuratowski measure of
noncompactness on $X_{i},$ then we can state the vectorial analogue of the
condition (h4*):

\begin{description}
\item[(H4)] for each $i=1,2,...,n,$ $\Phi _{i}=\Psi _{i}+\Theta _{i},$ where
$\Theta _{i}\left( \overline{U}\right) \left( t\right) \subset K_{i}$ for
a.a $t\in \left[ 0,a\right] ,$ $K_{i}$ being a compact set in $X_{i},$ and
there exist $\gamma _{ij}\in L_{+}^{1}\left( 0,a\right) $ $\left( 1\leq
j\leq n\right) ,$ such that for each countable set $C\subset \overline{U},$%
\begin{equation*}
\alpha _{i}\left( \Psi _{i}\left( C\right) \left( t\right) \right) \leq
\sum_{j=1}^{n}\gamma _{ij}\left( t\right) \alpha _{j}\left( C_{j}\left(
t\right) \right) ,\ \ \ \text{for a.a. ~}t\in \left[ 0,a\right] ,
\end{equation*}%
and
\begin{equation}
\rho \left( H\right) <1  \label{mm}
\end{equation}%
for the matrix
\begin{equation*}
H=2\left( \left\vert \mathcal{G}\right\vert \left\vert \widetilde{\gamma }%
\right\vert _{L^{1}\left( 0,a_{F}\right) }+\left\vert \gamma \right\vert
_{L^{1}\left( 0,a_{F}\right) }\right) .
\end{equation*}%
Here $\rho \left( H\right) $ is the spectral radius of $H$ and $\left\vert
\mathcal{G}\right\vert ,$ $\left\vert \gamma \right\vert _{L^{1}\left(
0,a_{F}\right) },$ $\left\vert \widetilde{\gamma }\right\vert _{L^{1}\left(
0,a_{F}\right) }$ are the matrices
\begin{eqnarray*}
\left\vert \mathcal{G}\right\vert &=&\left[ M_{i}\left\vert
G_{ij}\right\vert \right] _{1\leq i,j\leq n},\ \  \\
\left\vert \gamma \right\vert _{L^{1}\left( 0,a_{F}\right) } &=&\left[
M_{i}\left\vert \gamma _{ij}\right\vert _{L^{1}\left( 0,a_{F}\right) }\right]
_{1\leq i,j\leq n},\ \ \ \left\vert \widetilde{\gamma }\right\vert
_{L^{1}\left( 0,a_{F}\right) }=\left[ M_{i}\left\vert \widetilde{\gamma }%
_{ij}\right\vert _{L^{1}\left( 0,a_{F}\right) }\right] _{1\leq i,j\leq n}\ ,
\end{eqnarray*}%
where $\widetilde{\gamma }_{ij}\left( t\right) =\gamma _{ij}\left( t\right) $
for $t\in \left[ 0,a_{i}\right] ,\ \widetilde{\gamma }_{ij}\left( t\right)
=0 $ for $t\in (a_{i},a].$
\end{description}

\begin{theorem}
\label{thm3.1}Under the conditions \emph{\textbf{(H1)-(H4)},} the problem
\emph{(\ref{s1})} has a mild solution in $\overline{U}.$
\end{theorem}

\begin{proof}
The problem (\ref{s1}) is equivalent to the fixed point equation for the
nonlinear operator (\ref{N}) in $C\left( \left[ 0,a\right] ;X\right) ,$ $%
N=N_{1}+N_{2},$ where for each $i=1,2,...,n,$%
\begin{equation}
N_{2i}\left( u\right) \left( t\right) =\int_{0}^{t}T_{i}(t,s)\Phi _{i}\left(
u\right) \left( s\right) ds  \label{n2}
\end{equation}%
and%
\begin{eqnarray}
N_{1i}\left( u\right) \left( t\right)  &=&T_{i}\left( t,0\right) G_{i}\left(
N_{2}(u)\right)   \label{nn1} \\
&=&T_{i}\left( t,0\right) \sum_{j=1}^{n}G_{ij}\left( N_{2j}(u)\right)
\notag \\
&=&T_{i}\left( t,0\right) \sum_{j=1}^{n}G_{ij}\left( \chi _{a_{j}}\left(
N_{2j}\left( u\right) \right) \right) .  \notag
\end{eqnarray}%
Here $\chi _{a_{j}}:C\left( \left[ 0,a\right] ;X_{j}\right) \rightarrow
C\left( \left[ 0,a\right] ;X_{j}\right) $ is given by
\begin{equation*}
\chi _{a_{j}}\left( v\right) \left( t\right) =\left\{
\begin{array}{ll}
v\left( t\right)  & \text{if }t\in \left[ 0,a_{j}\right]  \\
v\left( a_{j}\right)  & \text{if }t\in (a_{j},a],%
\end{array}%
\right.
\end{equation*}%
for all $v\in C\left( \left[ 0,a\right] ;X_{j}\right) .$

We shall apply M\"{o}nch's continuation theorem in the Banach space $C\left( %
\left[ 0,a\right] ;X\right) ,$ to the open bounded set $\ U=U_{1}\times
U_{2}\times ...$ $\times U_{n},\ $where
\begin{equation*}
U_{i}=\{v\in C\left( \left[ 0,a\right] ;X_{i}\right) :\left\vert u_{i}\left(
t\right) \right\vert _{i}<R_{i}\left( t\right) \ \text{for }t\in \left[ 0,a%
\right] \}\ \ \ \left( 1\leq i\leq n\right)
\end{equation*}%
and to the element $u_{0}=0.$ Let $u=\lambda N\left( u\right) $ for some $%
u\in \overline{U}$ $\ $and$\ \lambda \in \left( 0,1\right) .$ From (\ref{ff1}%
), (\ref{n2}), we have%
\begin{equation}
\left\vert N_{2i}\left( u\right) \left( t\right) \right\vert _{i}\leq
M_{i}\int_{0}^{t}\left\vert \Phi _{i}\left( u\right) (s)\right\vert ds\leq
M_{i}\left\vert \delta _{i}\left( .\right) \psi _{i}\left( \left\vert
u_{i}\left( .\right) \right\vert _{i}\right) \right\vert _{L^{1}\left(
0,t\right) }  \label{n3}
\end{equation}%
Also, from (\ref{nn1}) and (\ref{n3}),
\begin{eqnarray}
\left\vert N_{1i}\left( u\right) \left( t\right) \right\vert _{i} &\leq
&M_{i}\sum_{j=1}^{n}\left\vert G_{ij}\right\vert \left\vert \chi
_{a_{j}}\left( N_{2j}\left( u\right) \right) \right\vert _{C\left( \left[ 0,a%
\right] ;X_{j}\right) }  \label{n1} \\
&\leq &M_{i}\sum_{j=1}^{n}\left\vert G_{ij}\right\vert M_{j}\left\vert
\delta _{j}\left( .\right) \psi _{j}\left( \left\vert u_{j}\left( .\right)
\right\vert _{j}\right) \right\vert _{L^{1}\left( 0,a_{j}\right) }  \notag \\
&\leq &r_{i}.  \notag
\end{eqnarray}%
Then, since $u=\lambda N\left( u\right) ,$ for each $t\in \left[ 0,a\right]
, $ one has%
\begin{equation*}
\left\vert u_{i}\left( t\right) \right\vert _{i}\leq \lambda \left(
r_{i}+M_{i}\left\vert \delta _{i}\left( .\right) \psi _{i}\left( \left\vert
u_{i}\left( .\right) \right\vert _{i}\right) \right\vert _{L^{1}\left(
0,t\right) }\right) =:c_{i}\left( t\right) .
\end{equation*}%
Next we follow the same argument as in the proof of Lemma \ref{lemma 2.2} in
order to show that
\begin{equation*}
c_{i}(t)<R_{i}(t)\ \ \ \text{for every\ \ }t\in \lbrack 0,a].
\end{equation*}%
To check condition (b) of Theorem \ref{th 1.1}, let $C\subset \overline{U}$
be countable and $C\subset \overline{\text{conv}}\left( \left\{ 0\right\}
\cup N\left( C\right) \right) .$ Then, for each $i,$
\begin{eqnarray}
\varphi _{i}\left( t\right) := &&\alpha _{i}\left( C_{i}\left( t\right)
\right) =\alpha _{i}\left( N_{i}\left( C\right) \left( t\right) \right)
\label{x1} \\
&\leq &\alpha _{i}\left( N_{1i}\left( C\right) \left( t\right) \right)
+\alpha _{i}\left( N_{2i}\left( C\right) \left( t\right) \right) ,\ \ \text{
}t\in \lbrack 0,a].  \notag
\end{eqnarray}%
Using (\ref{h1}), \textbf{(H4)} and (\ref{x1}) we obtain for a.e. $t\in
\lbrack 0,a]$%
\begin{eqnarray}
\alpha _{i}\left( N_{2i}\left( C\right) \left( t\right) \right) &\leq
&2\int_{0}^{t}M_{i}\alpha _{i}\left( \Phi _{i}\left( C\right) \left(
s\right) \right) ds  \label{x2} \\
&\leq &2M_{i}\int_{0}^{t}\sum_{j=1}^{n}\gamma _{ij}\left( s\right) \alpha
_{j}\left( C_{j}\left( s\right) \right) ds  \notag \\
&=&2M_{i}\int_{0}^{t}\sum_{j=1}^{n}\gamma _{ij}\left( s\right) \varphi
_{j}\left( s\right) ds.  \notag
\end{eqnarray}%
This, in view of (\ref{nn1}), yields%
\begin{eqnarray*}
\alpha _{i}\left( N_{1i}\left( C\right) \left( t\right) \right) &\leq
&M_{i}\alpha _{i}\left( \sum_{j=1}^{n}G_{ij}\left( \chi _{a_{j}}\left(
N_{2j}\left( C\right) \right) \right) \right) \\
&\leq &M_{i}\sum_{j=1}^{n}\left\vert G_{ij}\right\vert \alpha _{C_{j}}\left(
\chi _{a_{j}}\left( N_{2j}\left( C\right) \right) \right) ,
\end{eqnarray*}%
where $\alpha _{C_{j}}$ is the Kuratowski measure of noncompactness on $%
C\left( \left[ 0,a\right] ;X_{j}\right) .$ Furthermore, by (\ref{x2}) and
\textbf{(H4)}, we get
\begin{eqnarray*}
\alpha _{C_{j}}\left( \chi _{a_{j}}\left( N_{2j}\left( C\right) \right)
\right) &=&\max_{t\in \left[ 0,a\right] }\alpha _{j}\left( \chi
_{a_{j}}\left( N_{2j}\left( C\right) \right) \left( t\right) \right)
=\max_{t\in \left[ 0,a_{j}\right] }\alpha _{j}\left( \left( N_{2j}\left(
C\right) \right) \left( t\right) \right) \\
&\leq &2M_{j}\int_{0}^{a_{j}}\sum_{k=1}^{n}\gamma _{jk}\left( s\right)
\varphi _{k}\left( s\right) ds=2M_{j}\int_{0}^{a_{F}}\sum_{k=1}^{n}%
\widetilde{\gamma }_{jk}\left( s\right) \varphi _{k}\left( s\right) ds.
\end{eqnarray*}%
Then%
\begin{equation}
\alpha _{i}\left( N_{1i}\left( C\right) \left( t\right) \right) \leq
M_{i}\sum_{j=1}^{n}\left\vert G_{ij}\right\vert
2M_{j}\int_{0}^{a_{F}}\sum_{k=1}^{n}\widetilde{\gamma }_{jk}\left( s\right)
\varphi _{k}\left( s\right) ds.  \label{x3}
\end{equation}%
Now from (\ref{x1})-(\ref{x3}) we find%
\begin{equation*}
\varphi _{i}\left( t\right) \leq 2\sum_{j=1}^{n}M_{i}\left\vert
G_{ij}\right\vert \int_{0}^{a_{F}}\sum_{k=1}^{n}M_{j}\widetilde{\gamma }%
_{jk}\left( s\right) \varphi _{k}\left( s\right)
ds+2\int_{0}^{t}\sum_{k=1}^{n}M_{i}\gamma _{ik}\left( s\right) \varphi
_{k}\left( s\right) ds.
\end{equation*}%
If we denote%
\begin{eqnarray}
\gamma \left( t\right) &=&\left[ M_{i}\gamma _{ij}\left( t\right) \right]
_{1\leq i,j\leq n},\ \ \widetilde{\gamma }\left( t\right) =\left[ M_{i}%
\widetilde{\gamma }_{ij}\left( t\right) \right] _{1\leq i,j\leq n},\ \
\notag  \label{gammafi} \\
&& \\
\varphi \left( t\right) &=&\left[ \varphi _{1}\left( t\right) ,\varphi
_{2}\left( t\right) ,...,\varphi _{n}\left( t\right) \right] ^{\text{tr}},
\notag
\end{eqnarray}%
then the above inequalities for $i=1,2,...,n,$ can be put under the
vectorial form as%
\begin{equation}
\varphi \left( t\right) \leq 2\left\vert \mathcal{G}\right\vert
\int_{0}^{a_{F}}\widetilde{\gamma }\left( s\right) \varphi \left( s\right)
ds+2\int_{0}^{t}\gamma \left( s\right) \varphi \left( s\right) ds,\ \ \ \
t\in \left[ 0,a\right] .  \label{fff}
\end{equation}%
Finally we follow the same argument as in the proof of Theorem \ref{thm2.3},
in order to show that $\varphi \equiv 0$ on $\left[ 0,a\right] .$ The only
one difference is that for $t\in \left[ 0,a_{F}\right] ,$ from (\ref{fff}),
we have
\begin{equation*}
\varphi \left( t\right) \leq 2\left( \left\vert \mathcal{G}\right\vert
\left\vert \widetilde{\gamma }\right\vert _{L^{1}\left( 0,a_{F}\right)
}+\left\vert \gamma \right\vert _{L^{1}\left( 0,a_{F}\right) }\right)
\left\vert \varphi \right\vert _{L^{\infty }\left( 0,a_{F}\right)
}=H\left\vert \varphi \right\vert _{L^{\infty }\left( 0,a_{F}\right) },
\end{equation*}%
whence%
\begin{equation}
\left\vert \varphi \right\vert _{L^{\infty }\left( 0,a_{F}\right) }\leq
H\left\vert \varphi \right\vert _{L^{\infty }\left( 0,a_{F}\right) },
\label{fff1}
\end{equation}%
where by $\left\vert \varphi \right\vert _{L^{\infty }\left( 0,a_{F}\right)
} $ we mean the column matrix of entries $\left\vert \varphi _{i}\right\vert
_{L^{\infty }\left( 0,a_{F}\right) }.$ Then (\ref{fff1}) is equivalent to
the matrix inequality%
\begin{equation}
\left( I-H\right) \left\vert \varphi \right\vert _{L^{\infty }\left(
0,a_{F}\right) }\leq 0.  \label{fff2}
\end{equation}%
By (\ref{mm}), the entries of the matrix $\left( I-H\right) ^{-1}$ are
nonnegative, so in (\ref{fff2}) we can multiply to the left by $\left(
I-H\right) ^{-1}$ without changing the inequality, to obtain $\left\vert
\varphi \right\vert _{L^{\infty }\left( 0,a_{F}\right) }$ $\leq 0.$ Hence $%
\varphi \left( t\right) =0$ for all $t\in \left[ 0,a_{F}\right] .$ \noindent
The Gronwall's inequality implies that $\varphi (t)=0$ for all $t\in \lbrack
a_{F},a].$ Taking into account of (\ref{fff1}) and (\ref{gammafi}) we can
say that for each $i=1,2,...,n$ and for all $t\in \lbrack 0,a],$
\begin{equation*}
\varphi _{i}(t)=\alpha _{i}(C_{i}(t))=0,
\end{equation*}%
so $C_{i}(t)$ is relatively compact in $X_{i}$ and $C(t)=%
\prod_{i=1}^{n}C_{i}(t)$ is relatively compact in $X.$ \noindent Following
the same argument of the proof of Lemma \ref{II lemma} we have that $C$ is
equicontinuous, so we can say that the condition \textbf{(h4$^{0}$)} is
satisfied. \noindent On the other hand, by using \textbf{(H1)-(H3)}, as in
the proof of Lemma \ref{lemma 2.2} we can deduce \textbf{(h3$^{0}$)}.
\noindent Therefore Theorem \ref{th 2.1} provides the existence of at least
one mild solution $u=(u_{1},u_{2},...,u_{n})$ where $u_{i}\in \overline{U_{i}%
},$ $i=1,2,...,n.$
\end{proof}

\begin{remark}
\emph{In general, we have the matrix inequality }$\left\vert \widetilde{%
\gamma }\right\vert _{L^{1}\left( 0,a_F\right) }\leq \left\vert \gamma
\right\vert _{L^{1}\left( 0,a_F\right) }.$\emph{\ In particular, if }$%
a_{1}=a_{2}=...=a_{n}\ \left( =a_F\right) ,$\emph{\ i.e. }$\left[ 0,a_F%
\right] $\emph{\ is the support of }$F$\emph{\ with respect to all
variables, one has }$\gamma \left( t\right) =\widetilde{\gamma }\left(
t\right) $\emph{\ for all }$t\in \left[ 0,a_F\right] ,$\emph{\ which gives }$%
\left\vert \widetilde{\gamma }\right\vert _{L^{1}\left( 0,a_F\right)
}=\left\vert \gamma \right\vert _{L^{1}\left( 0,a_F\right) }$\emph{\ and }$%
H=2\left( \left\vert \mathcal{G}\right\vert +I\right) \left\vert \gamma
\right\vert _{L^{1}\left( 0,a_F\right) }.$
\end{remark}

To conclude, let us underline the combined contribution of the functions $%
\delta _{i},$ $\psi _{i},$ $R_{i},$ $\gamma _{ij}$ and numbers $M_{i}$ and $%
a_{i}$ to the conditions of Theorem \ref{thm3.1}. In particular, note the
different contribution of the support intervals $\left[ 0,a_{i}\right] ,$ $%
i=1,2,...,n,$ in realizing the assumptions (\ref{rr1}) and (\ref{mm}). As
smaller $a_{i}$ are, more chance for (\ref{rr1}), (\ref{mm}) exists. In the
limit case, where $a_{i}=0$ for all $i,$ that is for the classical Cauchy
problem, the conditions (\ref{rr1}) and (\ref{mm}) are trivially satisfied.

\bigskip

\noindent \textbf{Acknowledgements}\smallskip

\noindent The authors have been supported by the Gruppo Nazionale per
l'Analisi Matematica, la Probabilit\`{a} e le loro Applicazioni (GNAMPA) of
the Istituto Nazionale di Alta Matematica (INdAM) and by the National
Research Project GNAMPA 2013 \emph{Topological Methods for Nonlinear
Differential Problems and Applications}.

\end{document}